\documentclass[11pt,leqno]{amsart}
\usepackage{amsmath,amssymb,amsthm}
\usepackage{hyperref}

\usepackage{bbm}

\DeclareMathOperator{\diam}{diam\,}
\DeclareMathOperator{\co}{co}

\renewcommand{\geq}{\geqslant}
\renewcommand{\leq}{\leqslant}

\newcommand{\spann}{\operatorname{span}}

\newcommand{\1}[1]{\operatorname{\textbf{1}}}

\newtheorem{theorem}{Theorem}[section]
\newtheorem{lemma}[theorem]{Lemma}

\newtheorem{corollary}[theorem]{Corollary}
\theoremstyle{definition}

\theoremstyle{remark}
\newtheorem{remark}[theorem]{Remark}
\numberwithin{equation}{section}

\def\fnote#1{\footnote}

\def\ignora#1{}
\def\n3#1{\left\vert  \! \left\vert \! \left\vert \, #1 \, \right\vert \!
  \right\vert \! \right\vert }

\newcommand{\iten}{\ensuremath{\widehat{\otimes}_\varepsilon}}
\newcommand{\pten}{\ensuremath{\widehat{\otimes}_\pi}}

\begin{document}

\title{ Weak operator Daugavet property and weakly open sets in tensor product spaces }

\author{ Abraham Rueda Zoca }\address{Universidad de Granada, Facultad de Ciencias. Departamento de An\'{a}lisis Matem\'{a}tico, 18071-Granada
(Spain)} \email{ abrahamrueda@ugr.es}
\urladdr{\url{https://arzenglish.wordpress.com}}



\maketitle

\markboth{ABRAHAM RUEDA ZOCA}{WODP AND WEAKLY OPEN SETS IN TENSOR PRODUCT SPACES}

\begin{abstract}
We obtain new progresses about the diameter two property and the Daugavet property in tensor product spaces. Namely, the main results of the paper are:
\begin{itemize}
\item If $X^*$ has the WODP, then $X\iten Y$ has the DD2P for any Banach space $Y$.
\item If $X$ has the WODP, then $X\pten Y$ has the DD2P for any Banach space $Y$.
\item If $X^*$ and $Y^*$ have the WODP then $X\iten Y$ has the Daugavet property.
\end{itemize}
The above improve many results in the literature and establish progresses on some open questions.
\end{abstract}

\section{Introduction}\bigskip

The geometrical and topological properties of slices, weakly open sets and convex combinations of slices have been deeply studied because they have determined multiple properties of Banach spaces. In connection with the existence of such objects of small diameter we can highlight the characterisations of the \textit{Radon-Nikodym property (RNP), (convex) point of continuity property ((C)PCP)} or the \textit{strong regularity}. In the opposite extreme, the study of big slices, weakly open sets and convex combinations of slices have been analysed in connection with diameter two properties, octahedrality of the norm and the Daugavet property.

The study of the diameter two properties and the Daugavet properties in tensor product spaces have attracted the attention of many researchers in this century, leaving a vast literature on the topic \cite{aln,abr2011,blr2015ten, hlp17,kkw03,lange2020, llr,llr2,rtv21} together with many open questions (see e.g \cite{aln,llr,werner01}). In this line, probably the most long-standing open question is whether $X$ and $Y$ have the Daugavet property implies that $X\pten Y$ and/or $X\iten Y$ has the Daugavet proeprty \cite[Section 6, Question (3)]{werner01}. Moreover, in \cite[Section 5, (b)]{aln} it is asked how are diameter two properties in general preserved by tensor products.

Concerning the projective tensor product, the following results concerning the diameter two properties are known:
\begin{itemize}
\item If $X$ or $Y$ has the slice-D2P, then $X\pten Y$ has the slice-D2P \cite[Theorem 2.7]{aln}.
\item If $X$ and $Y$ have the SD2P, then $X\pten Y$ has the SD2P \cite[Corollary 3.6]{blr2015ten}. If one requires the SD2P just on $X$, then $X\pten Y$ may fail the SD2P \cite[Corollary 3.9]{llr2}.
\item If $K$ is an infinite compact Hausdorff space then $C(K)\pten X$ has the D2P for every non trivial Banach space $X$.
\end{itemize}
In view of the above results, it is clear that solid results exist for the slice-D2P and the SD2P in a projective tensor product; however, the study of the D2P seems to be missing. The main difficulty comes from the big advantage that slices provide in projective tensor product when compared with working with weakly open sets: since $B_{X\pten Y}:=\overline{\co}(B_X\otimes B_Y)$, an easy convexity argument implies that every slice of $B_{X\pten Y}$ must contain an elementary tensor of the form $x\otimes y$, for some $x\in B_X$ and $y\in B_Y$. This behaviour, no longer true for general non-empty weakly open subsets, explains the absense of results about the D2P in projective tensor products. Indeed, it is an open question to determine when $X\pten Y$ has the D2P (posed in an equivalent formulation in \cite[Question 4.2]{llr}), which is in turn a particular case of the above mentioned question \cite[Section 5, Question (b)]{aln}.

In the injective tensor product the situation is even worse because, as far as we know, there is no stability result of the diameter two properties by injective tensor product except the recent result that $X\iten Y$ may fail the SD2P even if $X$ enjoys the Daugavet property \cite[Theorem 3.1]{rueda24}. It is also knwon that $X\iten Y$ has the SD2P if $X$ is an infinite-dimensional $L_1$-predual \cite[Corollary 2.9]{belo06}, a result  based on the canonical identification $C(K)\iten X=C(K,X)$ \cite[Section 3.2]{ryan}. This time we can imagine the difficulty behind this problem when we see the injective tensor product as a space of bounded operators $X\iten Y\subseteq L(X^*,Y)$. With this point of view it seems clear that in order to study any diameter two property in the injective tensor product we need to handle and construct bounded linear operators, which is not an easy task. This observation, however, gives us the hint for the search of new results.

In \cite[Definition 5.2]{mr22}, the \textit{weak operator Daugavet property (WODP)} is introduced, as a weakening of the \textit{operator Daugavet property} introduced in \cite[Definition 4.1]{rtv21}, in order to obtain examples of projective tensor products with the Daugavet property. The WODP, which implies the Daugavet property \cite[Remark 5.3]{mr22}, is stable by taking projective tensor products in the sense that if $X$ and $Y$ have the WODP then $X\pten Y$ has the WODP \cite[Theorem 5.4]{mr22}.

In this paper we aim to get new results about the diameter two property in both injective and projective tensor products with the WODP. In Theorem~\ref{theo:injectivegeneral} we prove that $X\iten Y$ has the DD2P (in particular it has the D2P) if $X^*$ has the WODP and $Y$ is any Banach space $Y$. This can be applied, for instance, to get that $L_1(\mu)\iten Y$ has  the DD2P for any Banach space $Y$ if $\mu$ is an atomless measure (Corollary~\ref{coro:injeL1space}). We also prove in Theorem~\ref{theo:daugainjective} that $X\iten Y$ has the Daugavet property if both $X^*$ and $Y^*$ have the WODP. This result improves \cite[Theorem 1.1]{rtv21}, where it is proved the particular case when $X$ and $Y$ are $L_1$-spaces with the Daugavet property.

In Section~\ref{section:projective} we focus on studying the D2P and the DD2P in projective tensor product. First of all, we prove that if $X$ is an infinite-dimensional $L_1$-predual then $X\pten Y$ has the D2P for any Banach space $Y$, which improves \cite[Theorem 4.1]{abr2011}. Moreover, the main theorem in the section is Theorem~\ref{theo:wodpprojective}, where we prove that $X\pten Y$ has the DD2P whenever $X$ has the WODP and $Y$ is any Banach space. In addition to represent a progress on the open questions \cite[Question 4.2]{llr} and \cite[Section 5, (b)]{aln}, the above result improves \cite[Proposition 5.2]{lapi21}, where the authors obtained a weaker thesis under the same assumptions (see Remark~\ref{remark:projmejoraresul} for details).

\section{Notation and preliminary results}\bigskip

Throught the paper we will consider only real and non-trivial Banach spaces. We stand for $B_X$ and $S_X$ the closed unit ball and the unit sphere of the space $X$ respectively. We will denote by $X^*$ the topological dual of $X$. Given $A\subseteq X$, $\co(A)$ and $\spann(A)$ stands for the convex hull and the linear hull of the set $A$ in $X$.

Given two Banach spaces $X$ and $Y$ we will denote by $L(X,Y)$ the space of all the bounded operators from $X$ to $Y$. We will denote by $\mathcal F(X,Y)\subseteq L(X,Y)$ the space of finite-rank operators. Moreover, we will denote by $\mathcal F_{w^*-w}(X^*,Y)$ the space of those finite rank operator $T:X^*\longrightarrow Y$ which are weak$^*$-to-weak continuous.

\subsection{Tensor product spaces}

The projective tensor product of $X$ and $Y$, denoted by $X \pten Y$, is the completion of the algebraic tensor product $X \otimes Y$ endowed with the norm
$$
\|z\|_{\pi} := \inf \left\{ \sum_{n=1}^k \|x_n\| \|y_n\|: z = \sum_{n=1}^k x_n \otimes y_n \right\},$$
where the infimum is taken over all such representations of $z$. The reason for taking completion is that $X\otimes Y$ endowed with the projective norm is complete if, and only if, either $X$ or $Y$ is finite dimensional (see \cite[P.43, Exercises 2.4 and 2.5]{ryan}).

It is well known that $\|x \otimes y\|_{\pi} = \|x\| \|y\|$ for every $x \in X$, $y \in Y$, and that the closed unit ball of $X \pten Y$ is the closed convex hull of the set $B_X \otimes B_Y = \{ x \otimes y: x \in B_X, y \in B_Y \}$. Throughout the paper, we will use of both facts without any explicit reference.

Observe that the action of an operator $G\colon X \longrightarrow Y^*$ as a linear functional on $X \pten Y$ is given by
$$
G \left( \sum_{n=1}^{k} x_n \otimes y_n \right) = \sum_{n=1}^{k} G(x_n)(y_n),$$
for every $\sum_{n=1}^{k} x_n \otimes y_n \in X \otimes Y$. This action establishes a linear isometry from $ L(X,Y^*)$ onto $(X\pten Y)^*$ (see e.g. \cite[Theorem 2.9]{ryan}). All along this paper we will use the isometric identification $(X\pten Y)^*= L(X,Y^*)$ without any explicit mention.

Recall that a Banach space $X$ has the \textit{metric approximation property} (MAP)
if there exists a net $(S_\alpha)_\alpha$ in $\mathcal F(X,X)$ with $\Vert S_\alpha\Vert\leq 1$ for every $\alpha$ and such that
$S_\alpha (x) \to x$ for all $x \in X$. 

Recall that given two Banach spaces $X$ and $Y$, the
\textit{injective tensor product} of $X$ and $Y$, denoted by
$X \iten Y$, is the completion of $X\otimes Y$ under the norm given by
\begin{equation*}
   \Vert u\Vert_{\varepsilon}:=\sup
   \left\{
      \sum_{i=1}^n \vert x^*(x_i)y^*(y_i)\vert
      : x^*\in S_{X^*}, y^*\in S_{Y^*}
   \right\},
\end{equation*}
where $u=\sum_{i=1}^n x_i\otimes y_i$ (see \cite[Chapter 3]{ryan} for background).
Every $u \in X \iten Y$ can be viewed as an operator $T_u : X^* \rightarrow Y$ which is weak$^*$-to-weakly continuous. Under this point of view, the norm on the injective tensor product is nothing but the operator norm and, consequently, $X\iten Y$ can be seen as the norm-closure of $\mathcal F_{w^*-w}(X^*,Y)=\mathcal F_{w^*-w}(Y^*,X)$ in $L(X^*,Y)=L(Y^*,X)$. We will make use of this identification without further reference throughout the paper.

In order to define elements in the injective tensor product from finite-rank operators, the following lemma from \cite{ojapol07} will be particularly useful. This result asserts, roughly speaking, that every finite-rank operator between dual Banach spaces is an adjoint operator in a local way. We include the formal statement for easy reference since we will use this result several times in the text. 

\begin{theorem}\label{theo:preliojapol}\cite[Theorem 2.5]{ojapol07}
Let $X$ and $Y$ be Banach spaces, let $F\subseteq X^*$ be a finite-dimensional subspace and $\varepsilon>0$. If $T:X^*\longrightarrow Y^*$ is a finite-rank operator then there exists a finite-rank operator $S:Y\longrightarrow X$ satisfying the following assertions:
\begin{enumerate}
\item $\vert \Vert T\Vert-\Vert S\Vert\vert<\varepsilon$,
\item $S^*(X^*)=T(X^*)$,
\item $S^*(x^*)=T(x^*)$ holds for every $x^*\in F$ and,
\item $S^{**}(y^{**})=T^*(y^{**})$ holds for every $y^{**}\in Y^{**}$ for which $T^*(Y^{**})\in X$.
\end{enumerate}
\end{theorem}

Given two bounded operators $T\colon X\longrightarrow Z$ and $S\colon Y\longrightarrow W$, we can define an operator $T\otimes S\colon X\iten Y\longrightarrow Z\iten W$ by the action $(T\otimes S)(x\otimes y):=T(x)\otimes S(y)$ for $x\in X$ and $y\in Y$. It follows that $\Vert T\otimes S\Vert=\Vert T\Vert\Vert S\Vert$. It is known that if $T,S$ are linear isometries then $T\otimes S$ is also a linear isometry (c.f. e.g. \cite[Section 3.2]{ryan}). This fact is commonly known as ``the injective tensor product respects subspaces''.

It is also known that, given two Banach spaces $X$ and $Y$, the equality $(X\iten Y)^*=X^*\pten Y^*$ holds if either $X^*$ or $Y^*$ has the Radon-Nikodym Property (RNP) and either $X^*$ or $Y^*$ has the approximation property (AP) \cite[Theorem 5.33]{ryan}.

\subsection{Weak operator Daugavet property and related}

Here we will introduce the necessary notation from geometry of Banach spaces.

By a slice of the unit ball we denote a set of the form
$$S(B_X,x^*,\alpha):=\{x\in B_X: x^*(x)>\sup  (x^*(B_X))-\alpha\},$$
where $x^*\in X^*$ and $\alpha>0$. If $X$ is a dual space ($X=Y^*$) and $x^*\in Y\subseteq Y^{**}=X^*$, the abose set is said to be a \textit{$w^*$-slice}.

 We also consider a \textit{convex combination of slices of $B_X$} to be a set of the following form
$$\sum_{i=1}^n\lambda_i S_i,$$
where $S_1,\ldots, S_n$ are slices of $B_X$ and  $\lambda_1,\ldots, \lambda_n\in [0,1]$ satisfy that $\sum_{i=1}^n\lambda_i=1$.

Recall that a Banach space $X$ is said to have the \textit{Daugavet property} if given any rank-one bounded operator $T:X\longrightarrow X$ the equality
$$\Vert T+I\Vert=1+\Vert T\Vert$$
holds, where $I$ stands for the identity operator. This is the original formulation by I. Daugavet from \cite{dau}. In \cite{kssw} it was proved the following useful characterisation in geometric terms: a Banach space $X$ has the Daugavet property if, and only if, given any $x\in S_X$, any $\varepsilon>0$ and any slice $S$ of $B_X$ we can find $y\in S$ such that $\Vert x+y\Vert>2-\varepsilon$ \cite[Lemma 2.1]{kssw}. Observe that slices can be replaced with non-empty relatively weakly open subsets (indeed with convex combination of slices as shown in (the proof of)   \cite[Lemma 3]{shv}). In particular, if $X$ has the Daugavet property then every convex combination of slices (in particular every non-empty relatively weakly open subset of $B_X$) has diameter exactly two. This motivates us to introduce the diameter two properties we will need in this text.

Given a Banach space $X$, we say that $X$ has the:
\begin{itemize}
\item diameter two property (D2P) if every non-empty relatively weakly open subset of $B_X$ has diameter exactly two.
\item diametral diameter two property (DD2P) if, given any non-empty weakly open subset $W$ of $B_X$, any $x\in W\cap S_X$ and any $\varepsilon>0$ there exists $y\in W$ with $\Vert y-x\Vert>2-\varepsilon$.
\item strong diameter two property (SD2P) if every convex combination of slices of $B_X$ has diameter two.
\end{itemize}

The introduction of the D2P and the SD2P (together with another which we will not need) goes back to \cite{aln}, though the study of big weakly open sets and convex combinations of slices goes back earlier in the literature. On the other hand, the DD2P was introduced in \cite[Definition 2.1]{blr18}. The above three properties are known to be different each other (see e.g \cite{blr18}).

Going back to the Daugavet property, a longstanding open question from \cite{werner01} is whether $X\pten Y$ and/or $X\iten Y$ have/has the Daugavet property if both $X$ and $Y$ have the Daugavet property. In order to give a positive answer for the projective tensor product, in \cite{rtv21} and later in \cite{mr22} new geometric properties implying the Daugavet property were introduced in order to guarantee that $X\pten Y$ has the Daugavet property. Let us fix our attention in the following one from \cite[Definition 5.2]{mr22}: A Banach space $X$ is said to have the \textit{weak operator Daugavet property (WODP)} if, given $x_1,\ldots x_n\in S_X$, $\varepsilon>0$, a slice $S$ of $B_X$ and $x'\in B_X$, we can find $x\in S$ and $T\colon X\longrightarrow X$ with $\Vert T\Vert\leq 1+\varepsilon$, $\Vert T(x_i)-x_i\Vert<\varepsilon$ for every $i\in \{1,\ldots,n\}$ and $\Vert T(x)-x'\Vert<\varepsilon$. Examples of Banach spaces with the WODP are $L_1(\mu,X)$ if $\mu$ is an atomless $\sigma$-finite measure, $L_1$-preduals with the Daugavet property or the projective tensor product and the symmetric projective tensor products of Banach spaces with the WODP (see \cite[Section 5]{mr22} for details).

Observe that a perturbation argument allows, in the definition of WODP, to assume $\Vert T\Vert<1$. This evident observation will allow us to save a lot of homogeneity arguments in the future. Let us end with the following lemma, which will be employed in the proof of Theorem~\ref{theo:wodpprojective}. The proof will follow the lines of \cite[Lemmata 5.5 and 5.6]{mr22} because, in the end, is a mix of the ideas behind both results.

\begin{lemma}\label{lemma:proj}
Let $X$ be a Banach space with the WODP. Then, for every $k\in\mathbb N$, it follows that for every $x_1,\ldots, x_n\in S_X$, every non-empty relatively weakly open subsets $W_1,\ldots, W_k$ of $B_X$, every $y_1,\ldots, y_k\in B_X$ and every $\varepsilon>0$ there exists $z_i\in W_i, 1\leq i\leq k$, and a bounded linear operator $T:X\longrightarrow X$ satisfying that
\begin{enumerate}
\item $\Vert T\Vert< 1$,
\item $\Vert T(x_i)-x_i\Vert<\varepsilon$ holds for $1\leq i\leq n$ and,
\item $\Vert T(z_j)-y_j\Vert<\varepsilon$ holds for every $1\leq j\leq k$.
\end{enumerate}
\end{lemma}

\begin{proof}
The proof will follow the lines of \cite[Lemma 5.5]{mr22} and will be proved by induction on $k$.

The case $k=1$ follows from \cite[Lemma 5.6]{mr22} together with a normalisation argument. 

Assume that the result holds for $k\in\mathbb N$ and let us prove that it also holds for $k+1$.

In order to do so, take $W_1,\ldots, W_{k+1}\subseteq B_X$ non-empty weakly open sets, $x_1,\ldots, x_n\in S_X$, $y_1,\ldots, y_{k+1}\in B_X$ and $\varepsilon>0$. 

By the inductive step we can find $z_i\in W_i$ for $1\leq i\leq k$ and a bounded operator $T:X\longrightarrow X$ such that $\Vert T\Vert< 1$ and that
\begin{itemize}
\item $\Vert T(x_j)-x_j\Vert<\varepsilon$ holds for $1\leq j\leq n$ and $\Vert T(y_{k+1})-y_{k+1}\Vert<\varepsilon$.
\item $\Vert T(z_i)-y_i\Vert<\varepsilon$ holds for $1\leq i\leq k$.
\end{itemize} 

Now by the aforementioned \cite[Lemma 5.5]{mr22} we can define an operator $G:X\longrightarrow X$ with $\Vert G\Vert< 1$ and satisfying that:
\begin{itemize}
\item $\Vert G(x_j)-x_j\Vert<\varepsilon$ holds for $1\leq j\leq n$ and $\Vert G(z_i)-z_i\Vert<\varepsilon$ holds for $1\leq i\leq k$ and,
\item $\Vert G(z_{k+1})-y_{k+1}\Vert<\varepsilon$.
\end{itemize}
Doing $\Phi:=T\circ G:X\longrightarrow X$ we get the desired mapping. It is clear that $\Vert \Phi\Vert\leq \Vert T\Vert \Vert G\Vert<1$. On the other hand, using the same estimates as in the proof of \cite[Lemma 5.5]{mr22} it can be proved that $\Vert \Phi(x_j)-x_j\Vert<(2+\varepsilon)\varepsilon$ for every $1\leq j\leq n$ and $\Vert \Phi(z_i)-y_i'\Vert\leq (2+\varepsilon)\varepsilon$ holds for $1\leq i\leq k+1$. The arbitrariness of $\varepsilon>0$ concludes the proof of the lemma.
\end{proof}

\section{Injective tensor product}\label{section:injective}\bigskip

In this section we aim to obtain a result about the DD2P for the injective tensor product. We will do it in two steps. To begin with, we will prove the result assuming that we are dealing with finite-dimensional Banach spaces in order to exploit \cite[Theorem 5.33]{ryan}, and next we will make use of the fact that the injective tensor product respects subspaces in order to get the result in complete generality.

\begin{theorem}\label{theo:finitediminjective}
Let $X$ be a Banach space such that $X^*$ has the WODP and let $Y$ be any finite dimensional Banach space. Then $X\iten Y$ has the DD2P.
\end{theorem}

\begin{proof}
Let $\mathcal U$ be a non-empty weakly open subset of $B_{X\iten Y}$ and let $T\in \mathcal U\cap S_{X\iten Y}$. Let us prove that, given any $\varepsilon>0$, there exists $G\in\mathcal U$ such that $\Vert T-G\Vert>2(1-\varepsilon)-\varepsilon$. In order to do so we can assume that
$$\mathcal U:=\left\{G\in B_{X\iten Y}: \vert B_i(G)-B_i(T_0)\vert<\xi\ 1\leq i\leq n\right\},$$
where $B_1,\ldots, B_n\in (X\iten Y)^*=X^*\pten Y^*$ \cite[Theorem 5.33]{ryan}, $\xi>0$ and $T_0\in B_{X\iten Y}$. Since $T\in\mathcal U$ we infer $\vert B_i(T)-B_i(T_0)\vert<\xi$ holds for $1\leq i\leq n$, so we can find $\eta>0$ small enough to guarantee  $\vert B_i(T)-B_i(T_0)\vert<\xi-\eta$ holds for every $1\leq i\leq n$. 

On the other hand, since $B_i\in X^*\pten Y^*$ we can assume, up to a density argument, that $B_i:=\sum_{j=1}^m x_{ij}^*\otimes y_{ij}^*$ holds for certain $m\in\mathbb N$ and certain $x_{ij}^*\in X^*, y_{ij}^*\in Y^*$. Now select $\delta>0$ such that $\delta<\varepsilon$ and that $\delta\sum_{j=1}^m \Vert y_{ij}^*\Vert<\eta$ holds for every $1\leq i\leq n$.

Moreover find $x_0^*\in S_{X^*}$ and $y_0^*\in S_{Y^*}$ such that $T(x_0^*)(y_0^*)>1-\varepsilon$.

The assumption that $X^*$ has the WODP allows us to find $x^*\in \{z^*\in B_{X^*}: T(z^*)(y_0^*)>1-\varepsilon\}$ and a continuous operator $\phi:X^*\longrightarrow X^*$ such that $\Vert \phi\Vert< 1$, that $\Vert \phi(x_{ij}^*)-x_{ij}^*\Vert<\delta$ and $\Vert\phi(x^*)+x_0\Vert<\delta$. Define $T\circ\phi$, which is a finite-rank continuous operator with $\Vert T\circ\phi\Vert\leq \Vert T\Vert \Vert\phi\Vert<1$. An application of Theorem~\ref{theo:preliojapol} allows us to find an operator $S:Y^*\longrightarrow X$ such that $\Vert S^*\Vert\leq 1$ and $S^*(x_{ij}^*)=T\circ\phi(x_{ij}^*)$ holds for every $1\leq i\leq n, 1\leq j\leq m$ and $S^*(x^*)=T\circ\phi(x^*)$. Observe that $S^*:X^*\longrightarrow Y=Y^{**}$ is a finite-rank and $w^*-w^*$ continuous operator, so $S^*\in X\iten Y$.

Moreover, it follows that
\[\begin{split}
\vert B_i(S^*-T)\vert & =\left\vert \sum_{j=1}^m S^*(x_{ij}^*)(y_{ij}^*)-T(x_{ij}^*)(y_{ij}^*)\right\vert \\
& \leq \sum_{j=1}^m \vert S^*(x_{ij}^*)(y_{ij}^*)-T(x_{ij}^*)(y_{ij}^*)\vert \\
& =  \sum_{j=1}^m \vert (T\circ \phi)(x_{ij}^*)(y_{ij}^*)-T(x_{ij}^*)(y_{ij}^*)\vert\\
& \leq \sum_{j=1}^m \Vert T\Vert \Vert \phi(x_{ij}^*)-x_{ij}^*\Vert \Vert y_{ij}^*\Vert\\
& <\delta\sum_{j=1}^m \Vert y_{ij}^*\Vert<\eta.
\end{split}\]
Since $\vert B_i(T-T_0)\vert<\xi-\eta$ we get $ \vert B_i(S^*-T_0)\vert<\xi$, and the arbitrariness of $i$ implies that $G:=S^*\in \mathcal U$. Let us prove that $\Vert T-G\Vert>2(1-\varepsilon)-\varepsilon$. Indeed, 
\[\begin{split}\Vert T-G\Vert & \geq (T- G)(x^*)(y_0^*)\\
&  =T(x^*)(y_0^*)-S^*(x^*)(y_0^*)\\
&  =T(x^*)(y_0^*)-(T\circ\phi)(x^*)(y_0^*)\\
& =T(x^*)(y_0^*)+T(x_0^*)(y_0^*)-(T(x_0^*)(y_0^*)+T(\phi(x^*))(y_0^*))\\
& >2(1-\varepsilon)-\Vert T\Vert\Vert x_0^*+\phi(x^*)\Vert\Vert y_0^*\Vert
\\
& >2(1-\varepsilon)-\varepsilon.
\end{split}\]
as desired.
\end{proof} 
 
As stated before, it is time to remove the finite-dimensionality assumption in the above theorem to obtain the following result.

\begin{theorem}\label{theo:injectivegeneral}
Let $X$ be a Banach space such that $X^*$ has the WODP and let $Y$ be any Banach space. Then $X\iten Y$ has the DD2P.
\end{theorem}
 
\begin{proof}
Let $\mathcal U$ be a non-empty weakly open subset of $X\iten Y$ such that $\mathcal U\cap B_{X\iten Y}\neq \emptyset$ and select any $z\in \mathcal U\cap S_{X\iten Y}$ and $\varepsilon>0$, and let us prove that there exists $z'\in \mathcal U\cap B_{X\iten Y}$ such that $\Vert z-z'\Vert>2-\varepsilon$. In order to do so, we can assume with no loss of generality that $z\in X\otimes Y$, that is, that $z=\sum_{i=1}^n x_i\otimes y_i$ for certain $x_1,\ldots, x_n\in X, y_1,\ldots, y_n\in Y$. Set $E:= \spann\{y_1,\ldots, y_n\}$, and observe that $z\in\mathcal U\cap S_{X\iten E}$ (recall that the injective tensor product respects subspaces). Since $E$ is finite dimensional and $\mathcal U\cap B_{X\iten E}$ is a non-emtpy relatively weakly open subset of $B_{X\iten E}$, Theorem~\ref{theo:finitediminjective} implies that there exists $z'\in \mathcal U\cap B_{X\iten E}$ such that $\Vert z-z'\Vert_{X\iten E}>2-\varepsilon$. The fact that the injective tensor product respects subspaces implies that $z'\in \mathcal U\cap B_{X\iten Y}$ and $\Vert z-z'\Vert>2-\varepsilon$, as desired.
\end{proof} 

In the case that $X=L_1(\mu)$ we have the following characterisation.

\begin{corollary}\label{coro:injeL1space}
Let $(\Omega,\Sigma,\mu)$ a localizable $\sigma$-finite measure space. The following are equivalent:
\begin{enumerate}
\item $L_1(\mu)\iten X$ has the DD2P for every  Banach space $X$.
\item $L_1(\mu)$ has the DD2P.
\item $L_1(\mu)$ has the Daugavet property.
\item $\mu$ does not contain any atom.
\end{enumerate}
\end{corollary}

\begin{proof}
(1)$\Rightarrow$(2) is immediate, whereas (2)$\Leftrightarrow$(3)$\Leftrightarrow$(4) follows since $L_1(\mu)$ has the Daugavet property if, and only if, $\mu$ does not contain any atom, which is in turn equivalent to the fact that $B_{L_1(\mu)}$ does not have any denting point (see e.g. \cite[Lemma 3.2]{ahlp}). Finally (4)$\Rightarrow$(1) follows since $L_1(\mu)^*=L_\infty(\mu)$ has the WODP by \cite[Proposition 5.9]{mr22} if $\mu$ does not contain any atom. 
\end{proof}

\begin{remark}
Observe that in \cite[Theorem 3.1]{rueda24} it is proved that there exists a finite-dimensional Banach space $E$ such that $L_1^\mathbb C([0,1])\iten E$ fails the SD2P. Consequently, Theorem~\ref{theo:injectivegeneral} can not be improved to get the SD2P.
\end{remark}

In Theorem~\ref{theo:injectivegeneral} we have only made assumptions on one of the factors. In the following theorem we prove, however, that if we require that both $X^*$ and $Y^*$ have the WODP, we can improve the  thesis to get the Daugavet property in $X\iten Y$. Notice that we will not be able to make use of the identification of $(X\iten Y)^*=X^*\pten Y^*$ since \cite[Theorem 5.33]{ryan} does not apply here by the absense of any Radon-Nikodym property. However, we will make use of this in a local way.

\begin{theorem}\label{theo:daugainjective}
Let $X$ and $Y$ be two Banach spaces such that $X^*$ and $Y^*$ have the WODP. Then $X\iten Y$ has the Daugavet property.
\end{theorem}

\begin{proof}
Let $T\in S_{X\iten Y}$, $\varepsilon,\alpha>0$ and $B\in (X\iten Y)^*$ with $\Vert B\Vert=1$. Let us find $G\in S_{X\iten Y}$ such that $B(G)>1-\alpha$ and $\Vert T+G\Vert>2(1-\varepsilon)$.

Let us start by choosing $x_0^*\in S_{X^*}$ and $y_0^*\in S_{Y^*}$ such that $y_0^*(T(x_0^*))>1-\varepsilon$. Now select $S\in S_{X\iten Y}$ satisfying that $B(S)>1-\alpha$ and $S\in X\otimes Y$ (up to a density argument). Since $\Vert S\Vert=1$ we can, once again, find $u_0^*\in S_{X^*}$ and $v_0^*\in S_{Y^*}$ such that $v_0^*(S(u_0^*))>1-\varepsilon$. On the other hand, since $B(S)>1-\alpha$ then $B(S)>1-\alpha+\eta_0$ for certain $\eta_0>0$. 

By the condition $S\in X\otimes Y$ we get $S=\sum_{i=1}^n x_i\otimes y_i$ for certain $n\in\mathbb N$, $x_1,\ldots, x_n\in X$ and $y_1,\ldots, y_n\in Y$. Call $F:=\spann\{y_1,\ldots, y_n\}\subseteq Y$. Since the injective tensor product does respect subspaces we can see $S\in X\iten F\subseteq X\iten Y$. This allows us to note that
$$1-\alpha+\eta_0<B(S)=B_{|X\iten F}(S).$$
As $B_{|X\iten F}\in (X\iten F)^*=X^*\pten F^*$ \cite[Theorem 5.33]{ryan}, we conclude that $B_{|X\iten F}\in \overline{\co}(B_{X^*}\otimes B_{F^*})$, consequently we can find suitable $p\in\mathbb N, x_1^*,\ldots, x_p^*\in B_{X^*}\setminus\{0\}, y_1^*,\ldots, y_p^*\in B_{F^*}\setminus \{0\}$ and $\lambda_1,\ldots\lambda_p\in [0,1]$ with $\sum_{j=1}^p \lambda_j=1$ satisfying that

\begin{equation}\label{eq:aprofunfindim}
\left\Vert B_{|X\iten F}-\sum_{j=1}^p \lambda_i x_i^*\otimes y_i^*\right\Vert<\frac{\eta_0}{4}.
\end{equation}
Consequently,
\begin{equation}\label{eq:evalaproxfin}
\sum_{j=1}^p \lambda_j x_j^*\otimes y_j^*(S)>1-\alpha+\frac{3\eta_0}{4}.
\end{equation}
Select $0<\delta<\min\{\frac{\eta_0}{2},v_0^*(S(u_0^*))-(1-\varepsilon)\}$.

Since $X^*$ has the WODP we can find $x^*\in \{z^*\in B_{X^*}: y_0^*(T(z^*))>1-\varepsilon\}$ (which is a $w^*$-slice of $B_{X^*}$ since $T$ is $w^*-w$ continuous) and a bounded operator $\phi:X^*\longrightarrow X^*$ with $\Vert \phi\Vert<1$ and such that
\begin{equation}\label{eq:aproWODPpuntosfun}
\Vert \phi(x_j^*)-x_j^*\Vert<\delta, 1\leq j\leq m,
\end{equation}
and
\begin{equation}\label{eq:aproWODPpuntosnorma}
\Vert \phi(x^*)-u_0^*\Vert<\delta.
\end{equation}
Observe that $S\circ\phi:X^*\longrightarrow F$ is a finite-rank operator with $\Vert S\circ\phi\Vert<1$. An application of Theorem~\ref{theo:preliojapol} yields a finite-rank operator $U:F^*\longrightarrow X$ such that $\vert \Vert U\Vert-\Vert S\circ\phi\Vert \vert<1-\Vert S\circ\phi\Vert$ and $U^*(z^*)=S(\phi(z^*))$ holds for every $z^*\in \spann\{x_1^*,\ldots, x_p^*, x^*\}$. Now $V=U^*:X^*\longrightarrow F^{**}=F$ is a $w^*-w^*$ continuous, so in particular $V\in X\iten F\subseteq X\iten Y$. Moreover $\Vert V\Vert<1$ so we get that
\[
\begin{split}
B(V)& =B_{|X\iten F}(V)\geq \left(\sum_{j=1}^p \lambda_j (x_j^*\otimes y_j^*)\right)(V)-\left\Vert B_{|X\iten F}-\sum_{j=1}^p \lambda_j x_j^*\otimes y_j^*\right\Vert\\
& \mathop{>}\limits^{\mbox{\tiny{\eqref{eq:aprofunfindim}}}}\sum_{j=1}^p \lambda_j U^*(x_j^*)(y_j^*)-\frac{\eta_0}{4}=\sum_{j=1}^p\lambda_j S(\phi(x_j^*))(y_j^*)-\frac{\eta_0}{4}\\
& \geq \sum_{j=1}^p \lambda_j S(x_j^*)(y_j^*)-\sum_{j=1}^p\lambda_j\Vert \phi(x_j^*)-x_j^*\Vert-\frac{\eta_0}{4}\\
& \mathop{>}\limits^{\mbox{\tiny{\eqref{eq:aproWODPpuntosfun}}} }\left(\sum_{j=1}^p \lambda_j x_j^*\otimes y_j^*\right)(S)-\delta-\frac{\eta_0}{4}\mathop{>}\limits^{\mbox{\tiny{\eqref{eq:evalaproxfin}}} }1-\alpha+\frac{3\eta_0}{4}-\delta-\frac{\eta_0}{4}>1-\alpha.
\end{split}
\] 
Moreover
\[
\begin{split}
v_0^*(V(x^*))=v_0^*(S(\phi(x^*)))& \geq v_0^*(S(u_0))-\Vert \phi(x^*)-u_0^*\Vert\\
& \mathop{>}\limits^{\mbox{\tiny{\eqref{eq:aproWODPpuntosnorma}}} }v_0^*(S(u_0))-\delta>1-\varepsilon.
\end{split}
\]
Consequently we have found $V\in B_{X\iten Y}$ (indeed in $X\otimes Y$, in other words, $V:X^*\longrightarrow Y$ has finite rank) such that $B(V)>1-\alpha$ and there exists $x^*\in S_{X^*}$ such that $y_0^*(T(x^*))>1-\varepsilon$ and $v_0^*(V(x^*))>1-\varepsilon$. In the rest of the proof we will repeat the above argument in order to find $S\in B_{X\iten Y}$ and $y^*\in S_{Y^*}$ such that $B(S)>1-\alpha$ and that $y^*(T(x^*))>1-\varepsilon$ and $y^*(S(x^*))>1-\varepsilon$ holds, from where the inequality $\Vert T+S\Vert>2(1-\varepsilon)$ will be clear and the proof will be finished.

Since $V\in X\otimes Y$ write $V=\sum_{i=1}^m a_i\otimes b_i$ for certain $a_1,\ldots a_m\in X$ and $b_1,\ldots, b_m\in Y$, so we see $V\in E\iten Y$ where $E:=\spann\{a_1,\ldots, a_m\}\subseteq X$. As before, we can find $\xi>0$ satisfying $1-\alpha+\xi<B(V)=B_{|E\iten Y}(V)$.

As before observe that $B_{|E\iten Y}\in (E\iten Y)^*=E^*\pten Y^*$ and the fact that $\Vert B\Vert\leq 1$ allows us to find $q\in\mathbb N, a_1^*,\ldots, a_q^*\in B_{E^*}\setminus\{0\}$, $b_1^*,\ldots, b_q^*\in B_{Y^*}\setminus\{0\}$ and $\alpha_1,\ldots, \alpha_q\in [0,1]$ with $\sum_{k=1}^q\alpha_k=1$ such that
$$\left\Vert B_{|E\iten Y}-\sum_{k=1}^q \alpha_k a_k^*\otimes b_k^* \right\Vert<\frac{\xi}{4}.$$
Now we visualise $V:Y^*\longrightarrow E$ and note that
$$\sum_{k=1}^q \alpha_k a_k^*(V(b_k^*))>1-\alpha+\frac{3\xi}{4}.$$
Choose  $0<\gamma<\min\{\frac{\xi}{2},v_0^*(S(x^*))-(1-\varepsilon)\}$.

Since $Y^*$ has the WODP we can find an element $y^*\in \{w^*\in B_{Y^*}: w^*(T(x^*))>1-\varepsilon\}$ (which is a $w^*$-slice of $B_{Y^*}$ since $T:Y^*\longrightarrow X$ is $w^*-w$ continuous) and operator $\psi:Y^*\longrightarrow Y^*$ with $\Vert \psi\Vert<1$ and such that
$$\Vert \psi(b_k^*)-b_k^*\Vert<\gamma, 1\leq k\leq q$$
and
$$\Vert \psi(y^*)-v_0^*\Vert<\gamma.$$
Moreover since $V\circ\psi:Y^*\longrightarrow E$ is a finite-rank operator with $\Vert V\circ\psi\Vert<1$ we can find, with a new appeal to Theorem~\ref{theo:preliojapol}, a bounded operator $W:E^*\longrightarrow Y$ such that $\vert \Vert W\Vert-\Vert V\circ\psi\Vert\vert<1-\Vert V\circ\psi\Vert$ and such that $W^*(w^*)=V\circ\psi(w^*)$ holds for every $w^*\in \spann\{b_1^*,\ldots, b_q^*,y^*\}$. Our final operator $G:=W^*$ satisfies $B(G)>1-\alpha$ and $x^*(G(y^*))>1-\varepsilon$ by a similar argument to the previously exposed for $V$. \end{proof}

Now several comments are pertinent.

\begin{remark}\label{remark:mejoratrada}
Theorem~\ref{theo:daugainjective} improves \cite[Theorem 1.1]{rtv21}, where it is proved that $L_1(\Omega_1,\Sigma_1,\mu_1)\iten L_1(\Omega_2,\Sigma_2,\mu_2)$ has the Daugavet property if $\mu_i$ contains no atom for $i=1,2$. Indeed, in the above case the space $L_1(\Omega_i,\Sigma_i,\mu_i)^*=L_\infty(\Omega_i,\Sigma_i,\mu_i)$ has the WODP for $i=1,2$ by \cite[Proposition 5.9]{mr22} since $L_\infty(\Omega_i,\Sigma_i,\mu_i)$ are $L_1$-preduals enjoying the Daugavet property \cite[p. 78, Examples (b)]{werner01}. This means that Theorem~\ref{theo:daugainjective} covers \cite[Theorem 1.1]{rtv21}.
\end{remark}

\begin{remark}\label{remark:wodpuseful}
In \cite[Theorem 5.4]{mr22} it is proved that if both $X$ and $Y$ have the WODP then $X\pten Y$ has the WODP, which makes the WODP a potential tool to give an affirmative answer to the question whether $X\pten Y$ has the Daugavet property if both $X$ and $Y$ have the Daugavet property. Notice that Theorem~\ref{theo:daugainjective} proves that the WODP is also a potential tool to solve the analogous question for the injective tensor product. Indeed, observe that a Banach space $X$ has the Daugavet property if, and only if, given any $x^*\in S_{X^*}$, any $\varepsilon>0$ and any weak-star slice $S$ of $B_{X^*}$ there exists $y^*\in S$ such that $\Vert y^*-x^*\Vert>2-\varepsilon$ (which is a weak-star version of the Daugavet property) \cite[Lemma 2.1]{kssw}. Moreover, observe that in the proof of Theorem~\ref{theo:daugainjective} it is not needed that $X^*$ has the WODP, but just a weak-star version of it (since the slices used in the proof are actually weak-star slices as we have pointed out in the proof). 
\end{remark}

\section{Projective tensor product}\label{section:projective}\bigskip

In this section we get examples of spaces $X$ and $Y$ such that $X\pten Y$ has the D2P, which can be seen as partial answers to the question of when a projective tensor product has the D2P, posed in an equivalent formulation in \cite[Question 4.2]{llr}. Let us start with the next result, which is probably a simple observation. In \cite[Theorem 4.1]{abr2011} it is proved that $C(K)\pten X$ has the D2P if $K$ is infinite and $X$ is non trivial. In the next result we obtain a generalisation replacing $C(K)$ with a general infinite-dimensional $L_1$-predual space. The proof is based on the fact that the bidual of an $L_1$-predual space is a $C(K)$ space, the above mentioned \cite[Theorem 4.1]{abr2011}, and weak-star density argument allowed by the metric approximation property of any $C(K)$ space. We include a complete proof since, even though this result may be known by experts, we have not found any explicit reference to it.

\begin{theorem}\label{theo:predul1general}
Let $X$ be an infinite $L_1$-predual and let $Y$ be a non-zero any Banach space. Then $X\pten Y$ has the D2P.
\end{theorem}

\begin{proof}
Let $W$ be a non-empty weakly open subset of $B_{X\pten Y}$. We can assume with no loss of generality that
$$W=\left\{ z\in B_{X\pten Y}: \vert T_k(z-z_0)\vert<\alpha_0, 1\leq k\leq p\right\},$$ for suitable $T_k\in S_{L(X,Y^*)}$, $1\leq k\leq p$, $z_0\in X\pten Y$ and $\alpha_0>0$. Since $X$ is an infinite-dimensional $L_1$-predual then $X^{**}=C(K)$ for some infinite compact Hausdorff set $K$ \cite[Theorem 6.1]{linds}. Define the weak open set
$$U:=\left\{z\in B_{X^{**}\pten Y}: \vert T_k^{**}(z-z_0)\vert<\alpha_0, 1\leq k\leq p\right\},$$
which is non-empty because $W$ is non-empty and since the canonical embedding $X\pten Y\subseteq X^{**}\pten Y$ is isometric \cite[Lemma 2.3]{llr2}. Since $X^{**}=C(K)$ we get that $U$ has diameter two by \cite[Theorem 4.1]{abr2011}. Consequently, there are two elements $z_i:=\sum_{j=1}^{n_i}\alpha_{ij} x_{ij}^{**}\pten y_{ij}\in U\cap \co(B_{X^{**}}\otimes B_Y)$, $i=1,2$, such that $\Vert z_1-z_2\Vert>2-\varepsilon$. Since $X^{**}$ has the metric approximation property we can find an operator $T\in S_{L(X,Y^*)}$ such that
$$2-\varepsilon<T^{**}(z_1-z_2)=\sum_{j=1}^{n_1}\alpha_{1j} T^{**}(x_{1j}^{**})(y_{1j})-\sum_{j=1}^{n_2}\alpha_{2j} T^{**}(x_{2j}^{**})(y_{2j}),$$
(see, e.g. \cite[Lemma 2.3]{llr2}). Since $B_X$ is $w^*$-dense in $B_{X^{**}}$ then, for every $i\in\{1,2\}$ and every $j\in\{1,\ldots, n_i\}$, we can find a net $\{x_{ij}^s\}_s\subseteq B_X$ such that $x_{ij}^s\rightarrow^{w^*}x_{ij}^{**}$. Now notice that all the operators $T_j^{**}$ and $T^{**}$ are $w^*-w^*$-continuous. This implies that $T(x_{ij}^s)\rightarrow T^{**}(x_{ij}^{**})$ in the weak-star topology of $Y^*$, and so
$$T(x_{ij}^s)(y_{ij})\rightarrow T^{**}(x_{ij}^{**})(y_{ij})$$
holds for every $i,j$. Similarly, $T_k(x_{ij}^s)(y_{ij})\rightarrow T_k^{**}(x_{ij}^{**})(y_{ij})$ holds for every $i,j$ and $1\leq k\leq p$. Since $z_1,z_2\in U$ and $T^{**}(z_1-z_2)>2-\varepsilon$ we can find $s$ large enough so that $$\left\vert \sum_{j=1}^{n_i}\alpha_{ij} T_k(x_{ij}^s\otimes y_{ij})-T_k(z_0)\right\vert<\alpha_0$$
holds for every $1\leq k\leq p$ and $i=1,2$ (in other words, $\sum_{j=1}^{n_i}\alpha_{ij} x_{ij}^s\otimes y_{ij}\in W$) and that
$$\sum_{i=1}^{n_1}\alpha_{1j} T(x_{1j}^s)(y_{1j})-\sum_{i=1}^{n_2}\alpha_{2j} T(x_{2j}^s)(y_{2j})>2-\varepsilon,$$
which implies that $\diam(W)\geq 2-\varepsilon$. Since $\varepsilon>0$ was arbitrary the result follows.
\end{proof}

\begin{remark}
Theorem \ref{theo:predul1general} improves \cite[Theorem 4.1]{abr2011}, where it is proved that $\mathcal C(K)\pten X$ has the D2P for every Banach space $X$ and every infinite compact Hausdorff topological space $K$. Theorem \ref{theo:predul1general} also improves \cite[Proposition 2.11]{llr}, where it is proved that $c_0\pten X$ has the D2P for every Banach space $X$.
\end{remark}

Now it is time to prove that the WODP can be also used in order to get the DD2P in a projective tensor product. This is obtained in the following theorem.

\begin{theorem}\label{theo:wodpprojective}
Let $X$ be a Banach space with the WODP. Then $X\pten Y$ has the DD2P for every Banach space $Y$.
\end{theorem}

\begin{proof}
Let $\mathcal U\subseteq B_{X\pten Y}$ be a non-empty relatively weakly open subset of $B_{X\pten Y}$, set $z\in \mathcal U\cap S_{X\pten Y}$ and $\varepsilon>0$, and let us prove that there exists $w\in \mathcal U$ such that $\Vert z-w\Vert>2-\varepsilon$. In order to do so, we can find $n\in\mathbb N, x_1,\ldots, x_n\in B_X, y_1,\ldots, y_n\in B_Y$ and $\lambda_1,\ldots, \lambda_n\in (0,1]$ such that $\sum_{i=1}^n\lambda_i=1$, such that $z':=\sum_{i=1}^n\lambda_i x_i\otimes y_i\in\mathcal U$  and that
$$\left\Vert z-\sum_{i=1}^n \lambda_i x_i\otimes y_i \right\Vert<\frac{\varepsilon}{4}.$$
Since $\mathcal U$ is a non-empty relatively weakly open subset we can write
$$\mathcal U:=\left\{w\in B_{X\pten Y}: \left\vert  T_j(w-w_0)\right\vert<\alpha_0, 1\leq j\leq p \right\}$$
for some $T_1,\ldots, T_p\in (X\pten Y)^*=L(X,Y^*)$ with $\Vert T_j\Vert=1$ for $1\leq j\leq p$ and $\alpha_0>0$.
Since $z'\in\mathcal U$ this means that $\vert T_j(z'-w_0)\vert<\alpha_0$ holds for every $1\leq j\leq p$. Consequently, we can find $\eta>0$ small enough to guarantee that $\vert T_j(z'-w_0)\vert<\alpha_0-\eta$ holds for every $1\leq j\leq p$. By the definition of $z'$ the above means that
$$\left\vert \sum_{i=1}^n \lambda_i T_j(x_i)(y_i)-T_j(w_0)\right\vert<\alpha_0-\eta.$$
Given $1\leq i\leq n$ consider the following weakly open set
$$W_i:=\left\{z\in B_X: \left\vert T_j(z)(y_j)-T_j(x_i)(y_i)\right\vert<\frac{\eta}{n \lambda_j}\ \ \forall 1\leq j\leq p\right\},$$
which is a weakly open subset of $B_X$ containing $x_i$ for every $1\leq i\leq n$. Moreover, given any $z_i\in W_i$ we get that $\sum_{i=1}^n \lambda_i z_i\otimes y_i\in \mathcal U$. Indeed, given $1\leq i\leq n$ and $1\leq j\leq p$ we have
\[\begin{split}
\left\vert T_j\left(\sum_{i=1}^n \lambda_i z_i\otimes y_i\right)-T_j(w_0)\right\vert& 
=\left\vert \sum_{i=1}^n \lambda_i T_j(x_i)(y_i)-T_j(w_0)\right.\\
& \left. +\sum_{i=1}^n \lambda_i T_j(z_i-x_i)(y_i)\right\vert\\
& \leq \alpha_0-\eta +\sum_{i=1}^n \lambda_i \vert T_j(z_i)(y_i)-T_j(x_i)(y_i)\vert\\
& <\alpha_0-\eta+\sum_{i=1}^n \lambda_i \frac{\eta}{n\lambda_i}=\alpha_0,
\end{split}\]
as desired.

An application of Lemma~\ref{lemma:proj} implies the existence of $z_i\in W_i$ for every $1\leq i\leq n$ and a bounded operator $T:X\longrightarrow X$ such that $\Vert T\Vert\leq 1$, that $\Vert T(x_i)-x_i\Vert<\frac{\varepsilon}{8}$ and that $\Vert T(z_i)+x_i\Vert<\frac{\varepsilon}{8}$ holds for $1\leq i\leq n$.

Define $w:=\sum_{i=1}^n\lambda_i z_i\otimes y_i$, and let us prove that $w$ satisfies our purposes. On the one hand, by the above it follows that $w\in\mathcal U$, so let us estimate $\Vert z-w\Vert$. 

Since $\Vert z'\Vert>1-\frac{\varepsilon}{4}$ find $G\in S_{L(X,Y^*)}$ such that $G(z')=\sum_{i=1}^n \lambda_i G(x_i)(y_i)>1-\frac{\varepsilon}{4}$ and take $\phi:=G\circ T:X\longrightarrow X$. It is clear that $\Vert\phi\Vert\leq 1$. Now
\[
\begin{split}
\phi\left(\sum_{i=1}^n\lambda_i x_i\otimes y_i-\sum_{i=1}^n\lambda_i z_i\otimes y_i\right)& =\sum_{i=1}^n \lambda_i \phi(x_i-z_i)(y_i)\\
& =\sum_{i=1}^n \lambda_i G(T(x_i)-T(z_i))(y_i)\\
& =\sum_{i=1}^n \lambda_i G(2x_i+T(x_i)-x_i-(T(z_i)+x_i))(y_i)\\
& \geq 2\sum_{i=1}^n \lambda_i G(x_i)(y_i)-\sum_{i=1}^n\lambda_i (\Vert T(x_i)-x_i\Vert\\
& +\Vert T(z_i)+x_i\Vert)\\
& \geq 2\left(1-\frac{\varepsilon}{4}\right)-2\frac{\varepsilon}{8}=2-\frac{3\varepsilon}{4}.
\end{split}
\]
Consequently $\Vert z'-w\Vert\geq \phi(z'-w)\geq 2-\frac{3\varepsilon}{4}$, and since $\Vert z-z'\Vert<\frac{\varepsilon}{4}$ we infer that $\Vert z-w\Vert>2-\varepsilon$, as requested.
\end{proof}

Since $L_1$-preduals with the Daugavet property actually enjoy the WODP \cite[Proposition 5.9]{mr22} we get the following corollary.

\begin{corollary}
Let $X$ be an $L_1$-predual with the Daugavet property. Then $X\pten Y$ has the DD2P for every Banach space $Y$.
\end{corollary}

Let us conclude with a couple of remarks.

\begin{remark}\label{remark:projnomejorstrong}
Theorem~\ref{theo:wodpprojective} can not be improved to get the SD2P. Indeed, in \cite[Theorem 3.8]{llr2} it is proved that $L_\infty([0,1])\pten \ell_p^n$ fails the SD2P if $2<p<\infty$ and $n\geq 3$.
\end{remark}

\begin{remark}\label{remark:projmejoraresul}
Note that Theorem~\ref{theo:wodpprojective} improves \cite[Proposition 5.2]{lapi21} (which was in turn a strengthening of \cite[Remark 5.9]{rtv21}). In order to explain the way in which Theorem~\ref{theo:wodpprojective} improves \cite[Proposition 5.2]{lapi21}, let us introduce some notation. A Banach space $X$ is said to have the \textit{diametral local diameter two property DLD2P} if, given any slice $S$ of $B_X$, any $x\in S\cap S_X$ and any $\varepsilon>0$, there exists $y\in S$ such that $\Vert x-y\Vert>2-\varepsilon$.

In \cite[Proposition 5.2]{lapi21} it is proved that if $X$ has the WODP then $X\pten Y$ has the DLD2P for every non-zero space $Y$.  Our Theorem~\ref{theo:wodpprojective} improves the thesis to get the DD2P. Let us point out, however, that it is an open question whether the DLD2P and the DD2P are equivalent \cite[Question 4.1]{blr18}.
\end{remark}

\section*{Acknowledgements}  

The author thanks Johann Langemets for his contribution with interesting remarks and questions on the topic of the paper. The author also thanks Gin\'es L\'opez-P\'erez for fruitful conversations.

This work was supported by MCIN/AEI/10.13039/501100011033: grant PID2021-122126NB-C31, Junta de Andaluc\'ia: grant FQM-0185, by Fundaci\'on S\'eneca: ACyT Regi\'on de Murcia: grant 21955/PI/22 and by Generalitat Valenciana: grant CIGE/2022/97.

\end{document}